\documentclass{article}
\usepackage[utf8]{inputenc}
\usepackage{amsmath,amsfonts,amssymb,amsthm}
\usepackage{xurl}
\usepackage{hyperref}
\usepackage{cleveref}
\usepackage{biblatex}

\newtheorem{theorem}{Theorem}[section]
\newtheorem{lemma}[theorem]{Lemma}
\newtheorem{corollary}[theorem]{Corollary}

\newtheorem{question}{Question}
\newtheorem{definition}{Definition}
\newtheorem*{thmA}{Theorem A}

\DeclareMathOperator{\MLPS}{\mathsf{MLPS}}
\DeclareMathOperator{\MLRC}{\mathsf{MLRC}}
\DeclareMathOperator{\LRC}{\mathsf{LRC}}

\DeclareMathOperator*{\lcm}{lcm}
\DeclareMathOperator{\RR}{\mathbb{R}}
\DeclareMathOperator{\ZZ}{\mathbb{Z}}
\DeclareMathOperator{\ZZp}{\mathbb{Z}^+}
\DeclareMathOperator{\sgn}{sgn}
\DeclareMathOperator{\safe}{safe}
\DeclareMathOperator{\unsafe}{unsafe}

\title{Mixed thresholds in the\\Lonely Runner Conjecture}
\author{Alathea Jensen}

\begin{document}

\maketitle

\begin{abstract}
The Lonely Runner Conjecture states that if \(k+1\)
runners start at the same point on a unit-length circular track and run
with distinct constant speeds, then each runner is at some time at least \(1/(k+1)\)-distant from every other runner.  Equivalently, for every tuple of
\(k\) distinct positive integer speeds \(s_1,\ldots,s_k\), there is a real
number \(t\) such that
\[
    \|s_i t\|\geq \frac{1}{k+1}
    \qquad\text{for all }i.
\]

We introduce and study a version of the conjecture in which the
required distances may vary with \(i\).  For
\(\mathbf d=(d_1,\ldots,d_k)\in(0,1/2]^k\), let \(\MLPS_k\) be the set of
vectors such that, for every choice of distinct positive integer speeds
\(s_1,\ldots,s_k\), there is a real number \(t\) with
\[
    \|s_i t\|\geq d_i
    \qquad\text{for all }i.
\]
We give an exact characterization of \(\MLPS_2\).  We also use Fourier series for distance-threshold indicator functions to obtain an arithmetic progression summation formula and an exact
two-function integral formula for unequal thresholds.

\end{abstract}

\section{Introduction}

The Lonely Runner Conjecture is a classical problem in Diophantine approximation, originating in work of Wills and closely related to Cusick's view-obstruction formulation \cite{Wills1967,Cusick1974ViewObstructionGeometry,PerarnauSerraSurvey}.  In its classical formulation, it says that if \(n\) runners start at a common point on a circular track of length \(1\), each running with constant speed, then each runner is eventually lonely, meaning that there is a time at which that runner is at least \(1/n\) from every other runner, measured along the track.

After subtracting the speed of the runner under consideration, one obtains the following equivalent formulation.  Let
\[
    \|x\|:=\min_{m\in \ZZ}|x-m|
\]
denote the distance from \(x\) to the nearest integer.

Then the Lonely Runner Conjecture claims that for every ordered tuple of \(k\) nonzero integer speeds \(s_1,\ldots,s_k\), there exists a time \(t\in\RR\) such that
\[
    \|s_i t\|\geq \frac{1}{k+1}
    \qquad\text{for all }i=1,\ldots,k.
\]
This is the form of the conjecture that we will use throughout the paper.  We also say that a runner with speed \(s\) is \textbf{\(\delta\)-distant from the origin} or simply, \textbf{$\delta$-safe}, at time \(t\) if \(\|st\|\geq \delta\).

The Lonely Runner Conjecture has several equivalent formulations and has been studied from many points of view, including Diophantine approximation, view obstruction, chromatic numbers of distance graphs, covering radii of zonotopes, and nowhere-zero flows.  For a recent survey of these connections, see Perarnau and Serra \cite{PerarnauSerraSurvey}.  The case of seven total runners was proved by Barajas and Serra, using methods connected to distance graphs \cite{BarajasSerraSeven}.

Tao showed that the conjecture can be reduced to a finite computation, and Malikiosis, Santos, and Schymura later made that finite check substantially smaller \cite{TaoRemarks,MalikiosisSantosSchymuraFiniteChecking}.  Rosenfeld proved the eight-runner case \cite{RosenfeldEight}, Trakulthongchai proved the nine- and ten-runner cases \cite{TrakulthongchaiNineTen}, and Sungkawichai and Trakulthongchai announced the eleven-, twelve-, and thirteen-runner cases \cite{SungkawichaiTrakulthongchaiElevenTwelveThirteen}.

The usual Lonely Runner Conjecture asks for a common distance bound of \(1/(k+1)\).  This paper studies a variant in which the required distances are allowed to vary from runner to runner.\footnote{The author likes to imagine that some runners are more loquacious, so that the stationary runner doesn't feel lonely until they are further away, while other runners are more taciturn, so the stationary runner can feel lonely even when they are quite close.  This has nothing to do with the math, but it is an amusing extension of the loneliness metaphor.} Given an ordered vector
\[
    \mathbf d=(d_1,\ldots,d_k)\in (0,1/2]^k,
\]
we say that \(\MLRC(\mathbf d)\) holds if, for every choice of distinct positive integer speeds \(s_1,\ldots,s_k\), there exists a time \(t\in\RR\) such that
\[
    \|s_i t\|\geq d_i
    \qquad\text{for all }i=1,\ldots,k.
\]
We call this a mixed lonely runner statement.  Thus the classical Lonely Runner Conjecture for \(k+1\) total runners is the special case
\[
    \MLRC\left(\frac{1}{k+1},\ldots,\frac{1}{k+1}\right).
\]

The order of the coordinates in \(\mathbf d\) is not mathematically important, since the speeds may be relabeled.  It is nevertheless useful to treat \(\mathbf d\) as an ordered vector rather than a set, because this lets us assign a specific distance requirement to a specific speed during a proof.  We define the \(k\)-dimensional mixed loneliness parameter space by
\[
    \MLPS_k
    :=
    \left\{
        \mathbf d\in (0,1/2]^k:
        \MLRC(\mathbf d)\text{ holds}
    \right\}.
\]

One way to motivate this definition is through the view-obstruction formulation of the Lonely Runner Conjecture.  The motion of runners with speeds \(s_1,\ldots,s_k\) is described by the line
\[
    t\mapsto (s_1t,\ldots,s_kt)
\]
in the torus \(\RR^k/\ZZ^k\).  The usual Lonely Runner Conjecture asks whether this line must enter the cube
\[
    \left[\frac{1}{k+1},\frac{k}{k+1}\right]^k.
\]
Equivalently, in \(\RR^k\), the LRC asks whether every line through the origin in an integer direction must meet one of the translated copies of this cube.  The mixed version of LRC replaces the cube by an orthotope.  Namely, \(\MLRC(d_1,\ldots,d_k)\) asks whether every such line meets one of the regions
\[
    [d_1,1-d_1]\times \cdots \times [d_k,1-d_k]+\ZZ^k.
\]
Thus mixed loneliness is a generalization of the classical cube obstruction problem to orthotopes.

This generalization arose from an attempt to prove $\LRC(8)$ analytically.  The author was developing a formula for sums of indicator functions over arithmetic progressions of times, and found that the result can naturally be interpreted in the context of an indicator function with a different distance threshold.

In a forthcoming paper we prove the following mixed statements:
\[
    \MLRC\left(\frac14,\frac14,\frac18,\frac18\right),\qquad
    \MLRC\left(\frac14,\frac18,\frac18,\frac18,\frac18\right).
\]
These results were originally intended to facilitate case reductions for an analytic approach to proving \(\LRC(8)\).  While the eight-runner case and further cases are now proven by computer-assisted methods, these sorts of mixed statements are still interesting in their own right and may assist in future case reductions for other instances of the LRC.

This paper develops the basic language and tools for mixed loneliness.  We introduce a family of indicator functions that record when a runner is \(\delta\)-distant from the origin, and we use Fourier series to prove an arithmetic progression summation formula for these functions.  The summation formula describes how often a runner is \(\delta\)-distant from the origin along a finite arithmetic progression of times
\[
    t,\quad t+\frac1m,\quad t+\frac2m,\quad \ldots,\quad t+\frac{m-1}{m}.
\]
This type of summation is especially useful in proving cases of the LRC where a subset of speeds has a common divisor.

We also prove an exact formula for the integral of two indicator functions with different distance thresholds.  This generalizes a pairwise-intersection formula of Perarnau and Serra \cite{PerarnauSerraCorrelation}, who computed the measure of
\[
    \{t\in[0,1):\|v_i t\|<\delta\}
    \cap
    \{t\in[0,1):\|v_j t\|<\delta\}
\]
for two speeds \(v_i,v_j\) and a common distance threshold \(\delta\).  Our formula allows the two thresholds to differ.  Equivalently, it computes the measure of the set of times at which two runners are simultaneously \(\delta_1\)-distant and \(\delta_2\)-distant from the origin:
\[
    \{t\in[0,1):\|a t\|\geq \delta_1\}
    \cap
    \{t\in[0,1):\|b t\|\geq \delta_2\},
\]
where \(a,b\in\ZZ^+\) and \(\delta_1,\delta_2\in(0,1/2]\).

Our main theorem is the following.

\begin{thmA}
\label{thm:MLPS2}
The mixed loneliness parameter space \(\MLPS_2\) is exactly
\[
    \MLPS_2
    =
    \left\{
        (d_1,d_2)\in(0,1/2]^2:
        2d_1+d_2\leq 1
        \text{ and }
        d_1+2d_2\leq 1
    \right\}.
\]
\end{thmA}

Thus, in dimension two, the mixed parameter space is a convex polygon, rather than just the classical point \((1/3,1/3)\) arising from \(\LRC(3)\).

The paper is organized as follows.
\begin{itemize}
\item In Section~\ref{sec:mixed-loneliness}, we define mixed lonely runner statements and the spaces \(\MLPS_k\), and we explain the view-obstruction interpretation.
\item In Section~\ref{sec:safe-functions}, we introduce the safe and unsafe indicator functions and prove a formula for their sum over an arithmetic progression of times.  We also give an example of how the summation formula can be used in case reductions for the classical Lonely Runner Conjecture.
\item In Section~\ref{sec:MLPS2}, we prove an exact two-function integral formula for unequal thresholds and we characterize \(\MLPS_2\).
\item In Section~\ref{sec:future-work}, we discuss some questions about higher-dimensional mixed parameter spaces and briefly describe related work in progress.
\end{itemize}

\section{Mixed loneliness}
\label{sec:mixed-loneliness}

\subsection{Definitions}

We now give the formal definitions of mixed loneliness and the associated parameter spaces that will be used throughout the paper.

\begin{definition}[Mixed Lonely Runner Conjecture] For
\[
    \mathbf d=(d_1,\ldots,d_k)\in(0,1/2]^k,
\]
we write \(\MLRC(\mathbf d)\) for the following statement:

\begin{quote}
For every tuple of distinct positive integer speeds \((s_1,\ldots,s_k)\), there exists a time \(t\in\RR\) such that
\[
    \|s_i t\|\geq d_i
    \qquad\text{for each }i=1,\ldots,k.
\]
\end{quote}
\end{definition}

Thus the classical Lonely Runner Conjecture $\LRC(k+1)$ is the assertion
\[
    \MLRC\underbrace{\left(\frac{1}{k+1},\ldots,\frac{1}{k+1}\right)}_{k\textrm{ times}}.
\]
Since the speeds may be relabeled, \(\MLRC(\mathbf d)\) is unchanged by permuting the coordinates of \(\mathbf d\).  We nevertheless keep the ordered-vector notation, since in applications a particular threshold is often attached to a speed with a particular arithmetic property.

\begin{definition}[Mixed Loneliness Parameter Space]
For each \(k\geq 1\), define
\[
    \MLPS_k
    :=
    \left\{
        \mathbf d\in(0,1/2]^k:
        \MLRC(\mathbf d)\text{ holds}
    \right\}.
\]
We call this the \(k\)-dimensional mixed loneliness parameter space.
\end{definition}

The reason for excluding $d_i=0$ from these definitions is that when a threshold is zero, the corresponding runner has no effect on the problem, and so that problem should be studied in a lower dimension instead.  It might seem to be a trivial point, but in fact we can show that when $\MLPS_k$ is allowed to include points on coordinate hyperplanes, it is non-convex.  When points on coordinate hyperplanes are forbidden, however, the question of convexity is still open.  We discuss this more in \ref{sec:future-work}.

\subsection{View obstruction interpretation}

Given a speed vector
\[
    \mathbf s=(s_1,\ldots,s_k)\in\ZZ^k,
\]
the relative positions of the runners are described by the line
\[
    t\mapsto t\mathbf s
\]
in the torus \(\RR^k/\ZZ^k\).  For
\[
    \mathbf d=(d_1,\ldots,d_k),
\]
let
\[
    B(\mathbf d)
    :=
    [d_1,1-d_1]\times\cdots\times[d_k,1-d_k].
\]
Then \(\MLRC(\mathbf d)\) holds exactly when, for every set of distinct positive integer speeds \(s_1,\ldots,s_k\), the line \(t\mapsto t\mathbf s\) intersects \(B(\mathbf d)\) in \(\RR^k/\ZZ^k\).  Equivalently, in \(\RR^k\), every such line intersects the periodic set
\[
    B(\mathbf d)+\ZZ^k.
\]

In the classical case
\[
    \mathbf d=
    \left(\frac{1}{k+1},\ldots,\frac{1}{k+1}\right),
\]
the set \(B(\mathbf d)\) is the cube
\[
    \left[\frac{1}{k+1},\frac{k}{k+1}\right]^k.
\]
For general \(\mathbf d\), the cube is replaced by an orthotope.  Thus mixed loneliness is the orthotope version of the usual view-obstruction problem.

\subsection{Basic observations about the mixed loneliness parameter space}

It is immediate that \(\MLPS_k\) is downward closed in each coordinate, because if
\[
    (d_1,\ldots,d_k)\in\MLPS_k
    \qquad\text{and}\qquad
    0< e_i\leq d_i
\]
for every \(i\), then
\[
    (e_1,\ldots,e_k)\in\MLPS_k.
\]
Thus, once a point belongs to \(\MLPS_k\), so does the box beneath it, not including points on coordinate planes.

Another initial observation is that the spaces \(\MLPS_k\) contain a nontrivial neighborhood of the origin.  Using the trivial union-sum bound, every line in \(\RR^k\) whose direction is not parallel to a coordinate hyperplane intersects one of the translated cubes
\[
    \left[\frac{1}{2k},1-\frac{1}{2k}\right]^k+\ZZ^k.
\]
In the language of mixed loneliness, this implies
\[
    \left(0,\frac{1}{2k}\right]^k\subseteq \MLPS_k.
\]
In particular, each \(\MLPS_k\) is \(k\)-dimensional.

The classical Lonely Runner Conjecture, where proven, gives a point on the boundary of the parameter space.  A proof of \(\LRC(k+1)\) is exactly a proof that
\[
    \left(\frac{1}{k+1},\ldots,\frac{1}{k+1}\right)\in \MLPS_k.
\]
If it is in the parameter space, then this point must be on the boundary because the speed set \(\{1,2,\ldots,k\}\), which is the standard tight example for the classical Lonely Runner Conjecture, has the property that for every \(t\in\RR\), at least one of
\[
    \|t\|,\|2t\|,\ldots,\|kt\|
\]
is at most \(1/(k+1)\).  Therefore, for any choice of nonnegative numbers
\[
    \epsilon_1,\ldots,\epsilon_k>0,
\]
where at least one $\epsilon_i\neq 0$, we have
\[
    \left(\frac{1}{k+1}+\epsilon_1,\ldots,
          \frac{1}{k+1}+\epsilon_k\right)
    \notin \MLPS_k.
\]

Note however that this does not mean that the classical point is necessarily mean that the classical point can't lie in the convex hull of other points.  It may still be possible to increase some coordinates while decreasing others and remain inside \(\MLPS_k\).

We will mostly study \(\MLPS_k\) through indicator functions and their Fourier series expansions.  The next section introduces these functions and proves an arithmetic progression summation formula that will be useful in later applications.

\section{Indicator functions and arithmetic time progressions}
\label{sec:safe-functions}

\subsection{Safe and unsafe functions}

We begin by introducing midpoint-valued indicator functions for the times at which a runner lies outside a prescribed neighborhood of the origin.  The geometric origin of these distance-threshold sets goes back to the view-obstruction formulation, including the work of Cusick \cite{Cusick1974ViewObstructionGeometry} and later Chen and Cusick \cite{ChenCusickCubes}.  Perarnau and Serra also used indicator random variables
for the events
\[
    A_i=\{t\in[0,1):\|v_i t\|<\delta\}
\]
to study pairwise correlations among runners
\cite{PerarnauSerraCorrelation}.

Our approach is to analyze these distance-threshold functions using Fourier series.  Fourier methods have appeared in work on the Lonely Runner
Conjecture at least as early as Cusick and Pomerance's work on the view-obstruction problem \cite{CusickPomeranceViewObstructionIII}.  They also
play an important role in Tao's finite-checking reduction, where Fourier-analytic approximations to distance-threshold indicators are used \cite{TaoRemarks}.

For \(\delta\in(0,1/2]\), we say that a runner with speed \(s\) is \(\delta\)-distant from the origin at time \(t\) if
\[
    \|st\|\geq \delta.
\]
We also say that the runner is \(\delta\)-safe, or is in the \(\delta\)-safe zone.  When \(\delta\) is clear from context, we may simply say that the runner is safe.

The safe function is a midpoint-valued version of the indicator of the event \(\|st\|\geq\delta\).  Let \(s\in\ZZp\) and \(t\in\RR\).  For \(0<\delta<1/2\), define
\[
\safe(s,\delta,t)=
    \begin{cases}
      1, & \text{if }\|st\|>\delta,\\
      \frac12, & \text{if }\|st\|=\delta,\\
      0, & \text{if }\|st\|<\delta.
   \end{cases}
\]
Thus \(\safe(s,\delta,t)\) differs from the usual indicator function only at the boundary points \(\|st\|=\delta\).  The midpoint value is chosen so that the Fourier series below converges pointwise to the \(\safe\) function at the jump discontinuities.

For \(0<\delta<1/2\), we have
\[
    \safe(s,\delta,t)>0
    \quad\Longleftrightarrow\quad
    \|st\|\geq \delta.
\]
Consequently, positivity of a product of safe functions still detects the existence of a time at which all the corresponding distance requirements are met.

We define the endpoint cases where $\delta=0$ and $\delta=1/2$ by the limiting Fourier-series convention:
\[
    \safe(s,0,t)=1,
    \qquad
    \safe(s,1/2,t)=0.
\]
The second convention means that \(\safe(s,1/2,t)\) should be understood as the limiting midpoint-valued function rather than as the literal indicator of the event \(\|st\|\geq 1/2\).  We emphasize that this endpoint convention is a Fourier-series convention only.  When discussing \(\MLRC(\mathbf d)\) at boundary points of \((0,1/2]^k\), we always mean the original closed distance condition \(\|s_i t\|\geq d_i\).

We define the unsafe function by
\[
    \unsafe(s,\delta,t)=1-\safe(s,\delta,t).
\]
For \(0<\delta<1/2\), the function \(\unsafe(s,\delta,t)\) is the
midpoint-valued indicator of the event \(\|st\|\leq\delta\).  It is a rectangular pulse wave of period \(1/s\), mean value \(2\delta\), and
duty cycle \(2\delta\), and its Fourier series is
\[
\unsafe(s,\delta,t)
=
2\delta+
\frac{2}{\pi}
\sum_{j=1}^{\infty}
\frac{\sin(2\pi j\delta)}{j}\cos(2\pi s j t).
\]
Equivalently,
\[
\safe(s,\delta,t)
=
1-2\delta-
\frac{2}{\pi}
\sum_{j=1}^{\infty}
\frac{\sin(2\pi j\delta)}{j}\cos(2\pi s j t).
\]
These identities hold pointwise for \(0<\delta<1/2\), with the midpoint values at discontinuities.  The endpoint cases \(\delta=0\) and \(\delta=1/2\) will only be used with this limiting Fourier-series convention.

\subsection{Origin of mixed loneliness and a summation formula over arithmetic progressions of times}

We now explain how the concept of mixed loneliness arises naturally from the study of the classical uniform loneliness measure.  The starting point is the following simple observation.  If $m$ divides $s$, then
\[
    \|s(t+i/m)\|=\|st\|
    \qquad\text{for every }i\in\ZZ.
\]
Thus, if a runner whose speed is divisible by \(m\) is \(\delta\)-distant from the origin at time \(t\), then the runner is also \(\delta\)-distant from the origin at each of the times
\[
    t,\quad t+\frac1m,\quad t+\frac2m,\quad \ldots,\quad t+\frac{m-1}{m}.
\]
When \(m\nmid s\), this invariance no longer holds.  Nevertheless, the Fourier expansion of the safe function gives an exact formula for the sum of the safe values over this arithmetic progression of times.  The resulting formula expresses the sum of safe values with one distance threshold in terms of a single safe function with a different threshold.

We need two pieces of notation before proceeding.  First, for \(x\in\RR\), let $\langle x\rangle$ denote the centered fractional part of $x$, that is, the unique representative of \(x\) modulo \(1\) lying in    $\left(-\frac12,\frac12\right]$. Thus $|\langle x\rangle|=\|x\|$.

Second, we will use the sign function with the convention that \(\sgn(0)=0\).

\begin{lemma}
\label{lem:progression-safe-sum}
Let \(s,m\in\ZZp\), and let \(\delta\in(0,1/2]\).  Set
\[
    d=\gcd(s,m),
    \qquad
    \hat m=\frac{m}{d}.
\]
Then, for any \(t\in\RR\),
\[
\begin{aligned}
    \sum_{i=0}^{m-1}
    \safe&\left(s,\delta,t+\frac{i}{m}\right)\\
    \qquad &=
    m(1-2\delta)
    +
    d\,\sgn(\langle \hat m\delta\rangle)
    \Big(
        \safe\big(\lcm(s,m),\|\hat m\delta\|,t\big)
        -1+2\|\hat m\delta\|
    \Big).
\end{aligned}
\]
\end{lemma}

\begin{proof}
Write
\[
    d=\gcd(s,m),\qquad s=d\hat s,\qquad m=d\hat m,
\]
so that \(\gcd(\hat s,\hat m)=1\).  We first prove the identity for \(0<\delta<1/2\).  To justify the use of Fourier series, we can insert an
Abel factor \(r^j\), \(0<r<1\), in the Fourier expansion below, perform the finite summation, and then let \(r\to1^{-}\).  This avoids any issue with
interchanging the finite sum and the non-absolutely convergent Fourier series.

Using the Fourier expansion of the \(\safe\) function, we have
\begin{align*}
\sum_{i=0}^{m-1}
\safe&\left(s,\delta,t+\frac{i}{m}\right)\\
&=
\sum_{i=0}^{m-1}
\left(
1-2\delta
-\frac{2}{\pi}
\sum_{j=1}^{\infty}
\frac{\sin(2\pi j\delta)}{j}
\cos\left(2\pi sj\left(t+\frac{i}{m}\right)\right)
\right)\\
&=
m(1-2\delta)
-\frac{2}{\pi}
\sum_{j=1}^{\infty}
\frac{\sin(2\pi j\delta)}{j}
\underbrace{\sum_{i=0}^{m-1}
\cos\left(2\pi sj\left(t+\frac{i}{m}\right)\right)}_{*}.
\end{align*}

Now consider the inner sum marked with * above.  This can be simplified by converting to complex polar form and using the geometric series formula.  But there is also a more elegant geometric interpretation.

Consider the arguments to cosine that are being summed over.  Since $2\pi sj$ is a multiple of $2\pi$, and $\left\{t+\frac{i}{m}\right\}_{i=0}^{m-1}$ is an arithmetic progression evenly distributed over the interval $[0,1)$ modulo 1, we can conclude that the angles being summed over, modulo $2\pi$, will be evenly distributed about the unit circle in such a way that the corresponding points on the unit circle will have their centroid at the origin, except in the case that all of the angles are equal modulo $2\pi$, which happens exactly when $sj$ is a multiple of $m$.

In the cases where the centroid is at the origin (when $m\nmid sj$), the * expression is 0.  Hence, for the outer $\sum_{j=1}^\infty$ sum, we need only include values of \(j\) for which \(m\mid sj\).  Since
\[
    m\mid sj
    \quad\Longleftrightarrow\quad
    \hat m\mid \hat s j
    \quad\Longleftrightarrow\quad
    \hat m\mid j,
\]
only the terms \(j=\ell\hat m\) survive.  Therefore
\begin{align*}
\sum_{i=0}^{m-1}
\safe&\left(s,\delta,t+\frac{i}{m}\right)\\
&=
m(1-2\delta)
-\frac{2}{\pi}
\sum_{\ell=1}^{\infty}
\frac{m}{\ell\hat m}
\sin\left(2\pi\ell\hat m\delta\right)
\cos(2\pi s\ell\hat m t)\\
&=
m(1-2\delta)
-\frac{2d}{\pi}
\sum_{\ell=1}^{\infty}
\frac{1}{\ell}
\sin\left(2\pi\ell\hat m\delta\right)
\cos(2\pi s\ell\hat m t).
\end{align*}

Note that the final version of the summation above looks very similar to the Fourier expansion of a safe function.  To express it as such, we must ensure that what is being multiplied by $2\pi\ell$ inside the sine function is a constant in the interval $[0,1/2]$.

Recall that $\langle \hat m\delta\rangle\in(-1/2,1/2]$ is the centered fractional part of $\hat{m}\delta$, so that \(\hat m\delta\equiv \langle \hat m\delta\rangle\pmod 1\). Thus
\begin{align*}
    \sin\left(2\pi\ell\hat m\delta\right)
    &=
    \sin\left(2\pi\ell\langle \hat m\delta\rangle\right)\\
    &=
    \sgn(\langle \hat m\delta\rangle)\,
    \sin\left(2\pi\ell |\langle \hat m\delta\rangle|\right)\\
    &=
    \sgn(\langle \hat m\delta\rangle)\,
    \sin\left(2\pi\ell \|\hat m\delta\|\right).
\end{align*}

Thus
\begin{align*}
\sum_{i=0}^{m-1}
\safe&\left(s,\delta,t+\frac{i}{m}\right)\\
&=
m(1-2\delta)
+
d\,\sgn(\langle \hat m\delta\rangle)
\left(
-\frac{2}{\pi}
\sum_{\ell=1}^{\infty}
\frac{1}{\ell}
\sin\left(2\pi\ell \|\hat m\delta\|\right)
\cos(2\pi s \hat m \ell t)
\right).
\end{align*}
The expression in parentheses looks like the non-constant part of the Fourier expansion of
\[
    \safe\left(s\hat m,\|\hat m\delta\|,t\right).
\]
For \(0<\|\hat m\delta\|<1/2\), this is exactly the non-constant part of the Fourier expansion.  If \(\|\hat m\delta\|\in\{0,1/2\}\), the same expression is interpreted using the endpoint convention for \(\safe\), which agrees with the Fourier expansion.

Since \(s\hat m=\lcm(s,m)\), this gives
\begin{align*}
\sum_{i=0}^{m-1}
\safe&\left(s,\delta,t+\frac{i}{m}\right)\\
&=
m(1-2\delta)
+
d\,\sgn(\langle \hat m\delta\rangle)
\left(
    \safe\left(\lcm(s,m),\|\hat m\delta\|,t\right)
    -1+2\|\hat m\delta\|
\right),
\end{align*}
as was to be proven.

The endpoint cases \(\|\hat m\delta\|=0\) and \(\|\hat m\delta\|=1/2\) follow from the same Fourier-series convention used in the definition of \(\safe\), and can be verified directly from the formula above.
\end{proof}

Lemma~\ref{lem:progression-safe-sum} shows that summing a fixed distance threshold over an arithmetic progression can naturally produce a new
distance threshold.  This is one of the motivations for studying mixed loneliness.

In the next subsection we will explain how this summation can be used for case reductions of the classical Lonely Runner Conjecture, in a way reminiscent of the Prime Filtering Lemma of Barajas and Serra \cite{BarajasSerraSeven}.

To enable that application, let us look at a corollary that simplifies the Lemma~\ref{lem:progression-safe-sum} formula considerably when the threshold and the time progression are related by $\delta=1/n$ where $m\mid n$.

\begin{corollary}
\label{cor:safe-sum-classical-threshold}
Let \(s,n,m\in\ZZp\), with \(n>2\), \(m>1\), and \(m\mid n\).  Set
\[
    d=\gcd(s,m),
    \qquad
    c=\frac{n}{m}.
\]
Then, for every \(t\in\RR\),
\[
\sum_{i=0}^{m-1}
\safe\left(s,\frac1n,t+\frac{i}{m}\right)
=
\begin{cases}
m-\dfrac{2}{c},
& \text{if } cd\in\{1,2\},\\[8pt]
m-d+d\,
\safe\left(\lcm(s,m),\dfrac{1}{cd},t\right),
& \text{if } cd>2.
\end{cases}
\]
\end{corollary}
\begin{proof}
We apply Lemma~\ref{lem:progression-safe-sum} with $\delta=\frac1n$. Then, since $n=mc$,
\[
    \hat m\delta=\frac m d\cdot\frac 1 n=\frac m{dmc}=\frac 1 {cd}.
\]

If $cd=1$, then $\langle \hat m \delta \rangle=\left\langle 1\right\rangle= 0$.  So the formula from Lemma~\ref{lem:progression-safe-sum} collapses to $m(1-2\delta)=m-\frac{2m}{n}=m-\frac2 c$.

Otherwise, assuming $cd\geq2$, we have
\[\langle \hat m \delta \rangle=\left\langle \frac{1}{cd}\right\rangle= \frac{1}{cd}.\]

Thus

\[
\begin{aligned}
    \sum_{i=0}^{m-1}
    \safe&\left(s,\delta,t+\frac{i}{m}\right)\\
    \qquad &=
    m(1-2\delta)
    +
    d\,\sgn(\langle \hat m\delta\rangle)
    \Big(
        \safe\left(\lcm(s,m),\|\hat m\delta\|,t\right)
        -1+2\|\hat m\delta\|
    \Big)\\
    &=
    m\left(1-\frac2n\right)
    +
    d\Big(
        \safe\left(\lcm(s,m),\frac{1}{cd},t\right)
        -1+\frac2{cd}
    \Big)\\
    &=m-\frac{2m}n
    +
    d\,\safe\left(\lcm(s,m),\frac{1}{cd},t\right)
        -d+\frac2{c}\\
    &=m-d
    +
    d\,\safe\left(\lcm(s,m),\frac{1}{cd},t\right)
\end{aligned}
\]

In the case that $cd=2$, the threshold $1/cd$ is $1/2$, which by our endpoint convention has a safe value of 0 for any speed or time.  Hence the expression above becomes simply $m-d$ and since $cd=2$, $d=2/c$.
\end{proof}

Now we will present a concrete example of how the corollary can be applied to reduce cases in proofs of the classical LRC.

\subsection{Application to classical case reductions}
\label{subsec:case-reductions}

We now give an example of how Corollary~\ref{cor:safe-sum-classical-threshold} can be applied to reductions for the classical Lonely Runner Conjecture in any dimension.

Suppose that we are trying to prove \(\LRC(k+1)\), so that the desired distance threshold is \(1/(k+1)\).  Let
\[
    S=\{s_1,\ldots,s_k\}
\]
be a set of \(k\) distinct positive integer speeds.  Assume that \(M\subseteq S\) is the set of speeds divisible by \(k+1\), and write
\[
    R=S\setminus M.
\]
If \(M\) is nonempty, then after dividing the speeds in \(M\) by \(k+1\), any known lonely runner result for \(|M|+1\) total runners gives a time \(t\) such
that
\[
    \|st\|\geq \frac{1}{k+1}
    \qquad\text{for every }s\in M.
\]

Now consider the arithmetic progression of \(k+1\) times
\[
    t,\quad t+\frac{1}{k+1},\quad t+\frac{2}{k+1},
    \quad \ldots,\quad t+\frac{k}{k+1}.
\]
Every speed \(s\in M\) remains fixed modulo \(1\) along this progression, since
\(k+1\mid s\).  Therefore every speed in \(M\) is still \(1/(k+1)\)-safe at
each of these \(k+1\) times.

It remains to control the speeds in \(R\).  Let \(s\in R\), and set
\[
    d_s=\gcd(s,k+1).
\]
Applying Corollary~\ref{cor:safe-sum-classical-threshold} with
\(n=m=k+1\), hence \(c=1\), gives
\begin{align*}
\sum_{i=0}^{k}
\safe&\left(s,\frac1{k+1},t+\frac{i}{k+1}\right)\\
&=
\begin{cases}
k-1,
& \text{if } d_s\in\{1,2\},\\[8pt]
k+1-d_s+d_s\,
\safe\left(\lcm(s,k+1),\dfrac{1}{d_s},t\right),
& \text{if } d_s>2.
\end{cases}
\end{align*}
Equivalently,
\begin{align*}
\sum_{i=0}^{k}
\unsafe&\left(s,\frac1{k+1},t+\frac{i}{k+1}\right)\\
&=
\begin{cases}
2,
& \text{if } d_s\in\{1,2\},\\[8pt]
d_s\,
\unsafe\left(\lcm(s,k+1),\dfrac{1}{d_s},t\right),
& \text{if } d_s>2.
\end{cases}
\end{align*}
In particular,
\[
\sum_{i=0}^{k}
\unsafe\left(s,\frac1{k+1},t+\frac{i}{k+1}\right)
\leq
\max(2,d_s).
\]

Since \(\unsafe(s,1/(k+1),t)\) is greater than or equal to the indicator of the event
\[
    \left\{t:\|st\|<\frac1{k+1}\right\},
\]
the number of times in the above time progression at which \(s\) is not at least $1/(k+1)$-distant from the origin is at most \(\max(2,d_s)\).  Hence, by the union bound, the total number of times in the progression at which at least one speed in \(R\)
is not \(1/(k+1)\)-safe is at most
\[
    \sum_{s\in R}\max(2,\gcd(s,k+1)).
\]
Therefore, if
\[
    \sum_{s\in R}\max(2,\gcd(s,k+1)) < k+1,
\]
then at least one time in the progression avoids all of these bad events. At that time every speed in \(R\) is \(1/(k+1)\)-safe, while every speed in \(M\) is safe by construction.  Thus \(S\) satisfies the lonely runner bound
\(1/(k+1)\).

This gives a typical use of the progression formula.  Speeds divisible by some \(m\) are held fixed along a full residue progression, while the remaining speeds are controlled by their greatest common divisors with \(m\).  Thus Corollary~\ref{cor:safe-sum-classical-threshold} can be used for divisibility-based reductions similar to reductions such as the Prime Filtering Lemma of Barajas and Serra \cite{BarajasSerraSeven}.

\section{The two-dimensional mixed loneliness parameter space}
\label{sec:MLPS2}

\subsection{An exact two-function integral formula for unequal thresholds}

Before proving the characterization of $\MLPS_2$, we first provide an exact formula for the measure of time in $[0,1)$ at which a speed $s_1$ is $\delta_1$-safe and a speed $s_2$ is $\delta_2$-safe.  Equivalently, this is a formula for

\[\int_0^1\safe(s_1,\delta_1,t)\safe(s_2,\delta_2,t)dt.\]

In this lemma, we use $\{x\}$ to stand for the fractional part of $x$ in $[0,1)$, that is, $x-\lfloor x\rfloor$.

\begin{lemma}
\label{lem:two-safe-integral}
Let \(s_1,s_2\) be distinct positive integers, and let
\(\delta_1,\delta_2\in(0,1/2]\).  Set $\hat{s}_1=s_1/\gcd(s_1,s_2)$ and $\hat{s}_2=s_2/\gcd(s_1,s_2)$.
Let $B_2$ denote the second Bernoulli polynomial, $B_2(x)=x^2-x+1/6$. Then
\begin{align*}
\int_0^1
\safe&(s_1,\delta_1,t)\safe(s_2,\delta_2,t)\,dt\\
&=
(1-2\delta_1)(1-2\delta_2)
+
\frac{
B_2(\{\hat{s}_2\delta_1-\hat{s}_1\delta_2\})
-
B_2(\{\hat{s}_2\delta_1+\hat{s}_1\delta_2\})
}{
\hat{s}_1\hat{s}_2
}.
\end{align*}
\end{lemma}
\begin{proof}
First we will assume \(0<\delta_1,\delta_2<1/2\), and handle the endpoints later.  We use the Fourier expansion
\[
\safe(s,\delta,t)
=
1-2\delta
-
\frac{2}{\pi}
\sum_{j=1}^{\infty}
\frac{\sin(2\pi j\delta)}{j}
\cos(2\pi s j t).
\]
As before, the calculation below may be justified by inserting Abel factors \(r^j\), performing the computation for \(0<r<1\), and then letting \(r\to1^{-}\), so as to avoid issues with non-absolute convergence.

Multiplying the Fourier expansions for \(\safe(s_1,\delta_1,t)\) and \(\safe(s_2,\delta_2,t)\), then integrating over \([0,1]\), the constant terms give
\[
    (1-2\delta_1)(1-2\delta_2).
\]
All terms involving only one cosine integrate to zero.  Therefore
\[
\begin{aligned}
\int_0^1
\safe&(s_1,\delta_1,t)\safe(s_2,\delta_2,t)\,dt\\
&=
(1-2\delta_1)(1-2\delta_2)\\
&\qquad+
\frac{4}{\pi^2}
\sum_{j,\ell\in\ZZp}
\frac{\sin(2\pi j\delta_1)}{j}
\frac{\sin(2\pi \ell\delta_2)}{\ell}
\int_0^1
\cos(2\pi s_1jt)\cos(2\pi s_2\ell t)\,dt.
\end{aligned}
\]
By orthogonality,
\[
\int_0^1
\cos(2\pi s_1jt)\cos(2\pi s_2\ell t)\,dt
=
\begin{cases}
\frac12, & s_1j=s_2\ell,\\
0, & s_1j\neq s_2\ell.
\end{cases}
\]
Since $s_1=\gcd(s_1,s_2)\hat{s}_1$, $s_2=\gcd(s_1,s_2)\hat{s}_2$, and    $\gcd(\hat{s}_1,\hat{s}_2)=1$, we have
\[
    s_1j=s_2\ell\qquad\iff\qquad\hat{s}_1j=\hat{s}_2\ell.
\]
Thus the surviving terms are exactly those with
\[
    j=\hat{s}_2 q,
    \qquad
    \ell=\hat{s}_1 q,
    \qquad
    q\geq1.
\]
Hence
\[
\begin{aligned}
\int_0^1
\safe(s_1,\delta_1,t)\safe(s_2,\delta_2,t)\,dt
&=
(1-2\delta_1)(1-2\delta_2)\\
&\quad+
\frac{2}{\pi^2\hat{s}_1\hat{s}_2}
\sum_{q=1}^{\infty}
\frac{
\sin(2\pi \hat{s}_2 q\delta_1)
\sin(2\pi \hat{s}_1 q\delta_2)
}{q^2}.
\end{aligned}
\]
Using the product-to-sum identity for sines, that is,
\[
    \sin X\sin Y
    =
    \frac12\big(\cos(X-Y)-\cos(X+Y)\big),
\]
we obtain
\[
\begin{aligned}
\int_0^1
\safe&(s_1,\delta_1,t)\safe(s_2,\delta_2,t)\,dt\\
&=
(1-2\delta_1)(1-2\delta_2)\\
&\qquad+
\frac{1}{\pi^2\hat{s}_1\hat{s}_2}
\sum_{q=1}^{\infty}
\frac{
\cos\big(2\pi q(\hat{s}_2\delta_1-\hat{s}_1\delta_2)\big)
-
\cos\big(2\pi q(\hat{s}_2\delta_1+\hat{s}_1\delta_2)\big)
}{q^2}.
\end{aligned}
\]
Finally, the standard identity
\[
    \sum_{q=1}^{\infty}
    \frac{\cos(2\pi qx)}{q^2}
    =
    \pi^2 B_2(\{x\}),
    \qquad
    B_2(x)=x^2-x+\frac16,
\]
gives the formula in the lemma.

The endpoint cases where \(\delta_i=1/2\) follow by taking limits, using the Fourier-series endpoint convention in the definition of the \(\safe\) function.
\end{proof}

The formula from the lemma can also directly be translated into a formula for 

\[\int_0^1\unsafe(s_1,\delta_1,t)\unsafe(s_2,\delta_2,t) dt\]
by the fact that
\begin{align*}
\int_0^1&\unsafe(s_1,\delta_1,t)\unsafe(s_2,\delta_2,t) dt\\
&=\int_0^1\Big(1-\safe(s_1,\delta_1,t)-\safe(s_2,\delta_2,t)+\safe(s_1,\delta_1,t)\safe(s_2,\delta_2,t)\Big) dt\\
&=2\delta_1+2\delta_2-1+\int_0^1\safe(s_1,\delta_1,t)\safe(s_2,\delta_2,t)dt\\
&=4\delta_1\delta_2
+
\frac{
B_2(\{\hat{s}_2\delta_1-\hat{s}_1\delta_2\})
-
B_2(\{\hat{s}_2\delta_1+\hat{s}_1\delta_2\})
}{
\hat{s}_1\hat{s}_2
}.
\end{align*}

As such, our formula generalizes the formula of Perarnau and Serra's Proposition~8 in \cite{PerarnauSerraCorrelation}, which is for the equal threshold case.  Their formula is obtained geometrically and is written using maxes and mins of interval lengths, so the connection with our formula is not completely transparent.  In the equal-threshold case \(\delta_1=\delta_2=\delta\), however, our formula agrees with theirs.

\subsection{Proof of Theorem A}
As stated in the introduction, our main theorem is the following.
\begin{thmA}
The mixed loneliness parameter space \(\MLPS_2\) is exactly
\[
    \MLPS_2
    =
    \left\{
        (d_1,d_2)\in(0,1/2]^2:
        2d_1+d_2\leq 1
        \text{ and }
        d_1+2d_2\leq 1
    \right\}.
\]
\end{thmA}
\begin{proof}
We first prove the inclusion
\[
    \left\{
        (d_1,d_2)\in(0,1/2]^2:
        2d_1+d_2\leq1,\;
        d_1+2d_2\leq1
    \right\}
    \subseteq \MLPS_2.
\]

Let \(s_1,s_2\) be distinct positive integers.  Dividing both speeds by a common factor is equivalent to speeding up or slowing down time, so we may assume without loss of generality that
\[
    \gcd(s_1,s_2)=1.
\]

We now look at the times at which each runner is at position \(1/2\) on the track.  These are the centers of the safe intervals.  For speed \(s_1\), the centers are
\[
    \frac{\ZZ+1/2}{s_1},
\]
and for speed \(s_2\), the centers are
\[
    \frac{\ZZ+1/2}{s_2}.
\]

We claim that some center from the first set and some center from the second set have distance at most
\[
    \frac{1}{2s_1s_2}.
\]
To see this, write the two center sets with a common denominator:
\[
    \frac{\ZZ+1/2}{s_1}
    =
    \frac{1}{2s_1s_2}(2s_2\ZZ+s_2)
\]
and
\[
    \frac{\ZZ+1/2}{s_2}
    =
    \frac{1}{2s_1s_2}(2s_1\ZZ+s_1).
\]
Thus the possible signed differences between centers are
\[
    \frac{1}{2s_1s_2}
    \left(2(s_2z_2-s_1z_1)+s_2-s_1\right),
    \qquad z_1,z_2\in\ZZ.
\]
Since \(\gcd(s_1,s_2)=1\), the integer \(s_2z_2-s_1z_1\) can be chosen arbitrarily.  Hence the possible signed differences are
\[
    \frac{1}{2s_1s_2}(2\ZZ+s_2-s_1).
\]

If \(s_1\) and \(s_2\) are both odd, then \(1/2\) is a center for both families, so the two families have a common center.  Otherwise, since
\(\gcd(s_1,s_2)=1\), one of \(s_1,s_2\) is even and the other is odd.  Hence \(s_2-s_1\) is odd, and so
\[
    2\ZZ+s_2-s_1=2\ZZ+1.
\]
In particular, both $1$ and $-1$ occur as signed distances.  Thus some \(s_1\) safe center and some \(s_2\) safe center are at
 distance \(1/(2s_1s_2)\).

Now we will compare this distance between centers with the radii of the safe
intervals.  For speed \(s_1\), each safe interval has radius
\[
    \frac{1/2-d_1}{s_1}
    =
    \frac{1-2d_1}{2s_1}.
\]
For speed \(s_2\), each safe interval has radius
\[
    \frac{1/2-d_2}{s_2}
    =
    \frac{1-2d_2}{2s_2}.
\]
Therefore the two safe intervals overlap if
\[
    \frac{1-2d_1}{2s_1}
    +
    \frac{1-2d_2}{2s_2}
    \geq
    \frac{1}{2s_1s_2}.
\]
Multiplying by \(2s_1s_2\), this is equivalent to
\[
    s_2(1-2d_1)+s_1(1-2d_2)\geq1.
\]

We now verify this inequality throughout the claimed region for $\MLPS_2$.  The hypotheses
\[
    0\leq d_i\leq\frac12,
    \qquad
    2d_1+d_2\leq1,
    \qquad
    d_1+2d_2\leq1
\]
say that \((d_1,d_2)\) lies in the convex hull of
\[
    (0,0),\qquad
    \left(\frac12,0\right),\qquad
    \left(0,\frac12\right),\qquad
    \left(\frac13,\frac13\right).
\]
The function
\[
    s_2(1-2d_1)+s_1(1-2d_2)
\]
is linear in \((d_1,d_2)\), so its minimum on this polygon occurs at one of the vertices.  At the four vertices, its values are
\[
\begin{array}{c|c}
(d_1,d_2) & s_2(1-2d_1)+s_1(1-2d_2)\\
\hline
(0,0) & s_1+s_2\\[3pt]
(1/2,0) & s_1\\[3pt]
(0,1/2) & s_2\\[3pt]
(1/3,1/3) & \dfrac{s_1+s_2}{3}.
\end{array}
\]
Each of these values is at least \(1\).  The last one is at least \(1\)
because \(s_1\) and \(s_2\) are distinct positive integers, so $s_1+s_2\geq3$.
Hence
\[
    s_2(1-2d_1)+s_1(1-2d_2)\geq1
\]
throughout the whole region.  Thus the two safe intervals overlap, meaning there is some time $t$ when $\|s_1t\|\geq d_1$ and $\|s_2t\|\geq d_2$.

This proves the inclusion
\[
    \left\{
        (d_1,d_2)\in(0,1/2]^2:
        2d_1+d_2\leq1,\;
        d_1+2d_2\leq1
    \right\}
    \subseteq \MLPS_2.
\]

It remains to prove there are no other points in $\MLPS_2$.  Consider the speeds
\(s_1=1,s_2=2\).  On the interval \([0,1)\), speed \(s_1\) has one safe
interval, centered at \(t=1/2\), with radius
\[
    \frac12-d_1.
\]
Speed \(s_2\) has two safe intervals, centered at \(t=1/4\) and \(t=3/4\),
each with radius
\[
    \frac14-\frac{d_2}{2}.
\]
The distance from the \(s_1\) safe center to each of the two \(s_2\) safe centers is \(1/4\).  Since there are only these intervals, a simultaneous safe time can exist only if one of the \(s_2\) safe intervals
overlaps the \(s_1\) safe interval.  Using the radii, this requires
\[
    \left(\frac12-d_1\right)
    +
    \left(\frac14-\frac{d_2}{2}\right)
    \geq
    \frac14.
\]
Equivalently,
\[
    2d_1+d_2\leq1.
\]
Thus, if \(2d_1+d_2>1\), the speeds $s_1=1,s_2=2$ have no simultaneous
safe time, so \((d_1,d_2)\notin\MLPS_2\).

Similarly, applying the same argument to the speeds
\(s_1=2,s_2=1\) shows that membership in \(\MLPS_2\) also requires
\[
    d_1+2d_2\leq1.
\]
\end{proof}

\section{Open questions and future work}
\label{sec:future-work}

The first question is about products of the safe indicator functions.  It is not difficult to show that a finite product of safe functions with integer-valued speeds can be expressed using sums and differences of safe functions with speed 1 and different thresholds, except at the switching points.

For example, one can verify from manually taking intersections of the safe intervals that for all $t$ except for the switching points,
\[
\safe\left(2,\frac{1}{4},t\right)\cdot \safe\left(4,\frac{1}{4},t\right)\cdot \safe\left(3,\frac{1}{8},t\right)=\safe\left(1,\frac{1}{8},t\right)-\safe\left(1,\frac{3}{16},t\right).
\]

More generally, suppose \(f:[0,1)\to\{0,1\}\) is symmetric under
\(t\mapsto 1-t\), $f(0)=0$, and $f$ has finitely many switching points.
Let
\[
    0<\delta_1<\delta_2<\cdots<\delta_r\leq \frac12
\]
be the switching points in \((0,1/2]\), listed in increasing order.  Then at any time except possibly the
switching points,
\[
    f(t)
    =
    \safe(1,\delta_1,t)
    -
    \safe(1,\delta_2,t)
    +
    \safe(1,\delta_3,t)
    -
    \cdots
    +
    (-1)^{r+1}\safe(1,\delta_r,t).
\]
Products of safe functions have this form, since they are finite intersections of finite unions of intervals.

Of course, if we had a way to find all the times when a product of safe functions switches values, we would have already resolved the LRC.  Nevertheless, this sort of approach might have some utility, perhaps in finding bounds.

\begin{question}
Can we find useful bounds or expressions for products of safe
functions in terms of alternating sums of safe functions?
\end{question}

The two-dimensional parameter space \(\MLPS_2\) is a four-sided polygon cut out by the inequalities
\[
    d_1>0,
    \qquad
    d_2>0,
    \qquad
    2d_1+d_2\leq 1,
    \qquad
    d_1+2d_2\leq 1.
\]
Although two of its sides are open, it is nevertheless convex. It is natural to ask whether similar characterizations are true in higher dimensions.  The first basic structural question is whether the mixed parameter spaces have any convexity property.

One reason for excluding points with \(d_i=0\) from the parameter spaces was because we already know there are counterexamples to convexity if we include those points.  For example, we know
\[
    \left(\frac14,\frac14,\frac14\right)\in\MLPS_3,
\]
and suppose that we allow
\[
    \left(\frac13,\frac13,0\right)\in\MLPS_3.
\]
Then the line segment connecting those two points is
\[
\left\{
    \tau\left(\frac14,\frac14,\frac14\right)
    +(1-\tau)\left(\frac13,\frac13,0\right)
    :
    \tau\in[0,1]
\right\}.
\]
For points with
\[
    \frac47<\tau<1,
\]
the speed set \((1,3,2)\) suffices to show those points are not in \(\MLPS_3\), which can be checked by graphing the corresponding safe functions.

Nevertheless, it is still possible that after exclusion of points in the coordinate planes, $\MLPS_3$ could be convex, or a union of convex sets.  The author has not yet found any counterexamples.

A possible counterexample to convexity of $\MLPS_3$ might be the following.  The author has computational evidence, but not a complete proof, that $\left(\frac{7}{20},\frac{1}{5},\frac{13}{90}\right)$ and $\left(\frac{1}{20},\frac{1}{3},\frac{3}{10}\right)$ are likely in $\MLPS_3$.  Their midpoint, $\left(\frac{1}{5},\frac{4}{15},\frac{2}{9}\right)$, however, is not in $\MLPS_3$, which can be shown using the speed set $(2,1,3)$.

\begin{question}
For \(k\geq 3\), off of the coordinate planes, is \(\MLPS_k\) convex, or a finite union of convex sets?
\end{question}

The spaces \(\MLPS_k\) are coordinatewise downward closed, so in order to characterize them, we need only characterize their upper boundary.  It is clear that if the diagonal point
\[
    \left(\frac{1}{k+1},\ldots,\frac{1}{k+1}\right)
\]
is in the parameter space, then it is on the upper boundary, since it is impossible to increase any of its coordinates without leaving the space.

However, it is not clear whether it is in the convex hull of any other boundary points.

\begin{question}
Can the classical LRC point
\[
    \left(\frac{1}{k+1},\ldots,\frac{1}{k+1}\right)
\]
lie in the convex hull of other boundary points of \(\MLPS_k\)?
\end{question}

Another question concerns the obstructions that define these spaces.  In dimension two, the whole space is constrained by the two ordered speed pairs
\((1,2)\) and \((2,1)\).  It would be useful to know whether higher-dimensional spaces are similarly governed by a finite number of speed sets.  If true, this would also imply that the parameter spaces are finite unions of convex sets.

\begin{question}
For fixed \(k\), is \(\MLPS_k\) determined by finitely many speed sets?
\end{question}

There is a shifted version of the Lonely Runner Conjecture, in which the runners may have arbitrary initial positions; see, for example, Alc{\'a}ntara, Criado, and Santos \cite{AlcantaraCriadoSantosShiftedLRC}.  It would be interesting to combine shifted and mixed LRC.  Given that the projection of an orthotope is also a zonotope, it seems plausible that the zonotopal covering-radius methods used for the shifted conjecture could be adapted tov study mixed shifted parameter spaces.

\begin{question}
Can the zonotopal covering-radius methods used for shifted lonely runner
problems be adapted to mixed thresholds?
\end{question}

Finally, Kravitz introduced a loneliness spectrum conjecture
\cite{KravitzBarelyVeryLonely}, and Giri and Kravitz later studied lonely runner spectra as objects in their own right \cite{GiriKravitzLonelyRunnerSpectra}. Mixed loneliness suggests a higher-dimensional version of spectra.

\begin{question}
Can the mixed parameter spaces \(\MLPS_k\) be described in terms of suitable
mixed spectra of speed sets?
\end{question}

\printbibliography

\end{document}